\documentclass[a4paper,10pt]{article}
\usepackage[utf8x]{inputenc}
 \usepackage[T1]{fontenc} 
 \usepackage[finnish,english]{babel}

\usepackage{amsmath}
\usepackage{amsthm}
\usepackage{amssymb}
\usepackage{mathtools}
\usepackage{epsfig}
\usepackage{psfrag}
\usepackage{graphicx} 
\usepackage{caption} 
\usepackage{subcaption}

\newtheorem{theorem}{Theorem}[section]

\newtheorem{remark}[theorem]{Remark}

\title{Hyperbolic triangular buildings without periodic planes of genus two}
\author{Riikka Kangaslampi, Alina Vdovina}

\begin{document}
     
\maketitle

\begin{abstract}
  We study surface subgroups of groups acting simply transitively on
  vertex sets of certain hyperbolic triangular buildings. The study is
  motivated by Gromov's famous surface subgroup question: Does every
  one-ended hyperbolic group contain a subgroup which is isomorphic to
  the fundamental group of a closed surface of genus at least 2? In
  \cite{KV2} and \cite{CKV} the authors constructed and classified all
  groups acting simply transitively on the vertices of hyperbolic
  triangular buildings of the smallest non-trivial thickness. These
  groups gave the first examples of cocompact lattices acting simply
  transitively on vertices of hyperbolic triangular Kac-Moody
  buildings that are not right-angled. Here we study surface subgroups
  of the 23 torsion free groups obtained in \cite{KV2}.  With the help
  of computer searches we show, that in most of the cases there are no
  periodic apartments invariant under the action of a genus two
  surface. The existence of such an action implies the existence of a
  surface subgroup, but it is not known, whether the existence of a
  surface subgroup implies the existence of a periodic
  apartment. These groups are the first candidates for groups that
  have no surface subgroups arising from periodic apartments.

\end{abstract}

\section{Introduction}

In \cite{KV2} the authors classified all torsion-free groups acting
simply transitively on the vertices of hyperbolic triangular buildings
of the smallest non-trivial thickness. They constructed the groups
with the polygonal presentation method introduced in
\cite{Vdovina}. As a result, they obtain 23 non-isomorphic groups,
each defined by 15 generators $x_1,x_2,\ldots, x_{15}$ and 15 cyclic
relations, each of them of the form $x_ix_jx_k=1$, where not all the
indices are the same. The underlying hyperbolic building is the
universal cover of the polyhedron glued together from 15 geodesic
triangles with angles $\pi/4$ and with the letters from the relations
written on the boundary. In constructing the polyhedron the sides of
the triangles with the same labels are glued together, respecting the
orientation. For example, the presentations $T_1$, $T_3$, $T_9$ and $T_{21}$
obtained in \cite{KV2} are given in Table \ref{presT}. These will be
used later as examples.

\begin{table}[h!]
\centering
    \begin{tabular}{| l | l | l | l |}
    \hline
   $T_1$ & $T_{3}$ & $T_9$ & $T_{21}$\\
    \hline
$(x_1,x_1,x_{10})$&$(x_1, x_{1},x_{10})$ &$(x_1, x_{1},x_{10})$&$(x_1, x_{5},x_{2})$\\
$(x_1,x_{15},x_2)$&$(x_1, x_{15}, x_{2})$&$(x_1, x_{15},x_{2})$&$(x_4, x_{13}, x_{11})$\\
$(x_2,x_{11},x_9)$&$(x_2, x_{11}, x_{3})$&$(x_2, x_{11},x_{4})$&$(x_1, x_{6}, x_{4})$\\
$(x_2,x_{14},x_3)$&$(x_2, x_{14}, x_{5})$&$(x_2, x_{14},x_{6})$&$(x_5, x_{9}, x_{10})$\\
$(x_3,x_7,x_4)$&$(x_3, x_{7}, x_{4})$&$(x_3, x_{5},x_{9})$&$(x_1, x_{3}, x_{13})$\\
$(x_3,x_{15},x_{13})$&$(x_3, x_{15}, x_{8})$&$(x_3, x_{8},x_{7})$&$(x_5, x_{13}, x_{9})$\\
$(x_4,x_8,x_6)$&$(x_4, x_{8}, x_{9})$&$(x_3, x_{10},x_{13})$&$(x_2, x_{7}, x_{10})$\\
$(x_4,x_{12},x_{11})$&$(x_4, x_{12}, x_{12})$&$(x_4, x_{8},x_{5})$&$(x_6, x_{9}, x_{8})$\\
$(x_5,x_5,x_8)$&$(x_5, x_{9}, x_{6})$&$(x_4, x_{14},x_{14})$&$(x_2, x_{12}, x_{15})$\\ 
$(x_5,x_{10},x_{12})$&$(x_5, x_{13}, x_{13})$&$(x_5, x_{10},x_{12})$&$(x_6, x_{11}, x_{10})$\\
$(x_6,x_6,x_{14})$&$(x_6, x_{8}, x_{11})$&$(x_6, x_{7},x_{12})$&$(x_3, x_{11}, x_{14})$\\
$(x_7,x_7,x_{12})$&$(x_6, x_{10}, x_{13})$&$(x_6, x_{15},x_{9})$&$(x_7, x_{8}, x_{15})$\\
$(x_8,x_{13},x_{9})$&$(x_7, x_{9}, x_{14})$&$(x_7, x_{8},x_{11})$&$(x_3, x_{14}, x_{8})$\\ 
$(x_9,x_{14},x_{15})$&$(x_{7}, x_{10}, x_{12})$&$(x_9, x_{15},x_{13})$&$(x_{7}, x_{14}, x_{12})$\\
$(x_{10},x_{13},x_{11})$&$(x_{11}, x_{15}, x_{14})$&$(x_{11}, x_{12},x_{13})$&$(x_4, x_{12}, x_{15})$\\
    \hline
    \end{tabular}
      \caption{Presentations $T_1$, $T_3$, $T_{9}$ and $T_{21}$ from \cite{KV2}.}
      \label{presT}
    \end{table}

%\begin{table}[h!]
%  \centering
%  \begin{tabular}{| l | l |}
%    \hline
%    $T_1$ & $T_{21}$\\
%    \hline
%    $(x_1,x_1,x_{10})$&$(x_1, x_{5},x_{2})$\\
%    $(x_1,x_{15},x_2)$&$(x_4, x_{13}, x_{11})$\\
%    $(x_2,x_{11},x_9)$&$(x_1, x_{6}, x_{4})$\\
%    $(x_2,x_{14},x_3)$&$(x_5, x_{9}, x_{10})$\\
%    $(x_3,x_7,x_4)$&$(x_1, x_{3}, x_{13})$\\
%    $(x_3,x_{15},x_{13})$&$(x_5, x_{13}, x_{9})$\\
%    $(x_4,x_8,x_6)$&$(x_2, x_{7}, x_{10})$\\
%    $(x_4,x_{12},x_{11})$&$(x_6, x_{9}, x_{8})$\\
%    $(x_5,x_5,x_8)$&$(x_2, x_{12}, x_{15})$\\ 
%    $(x_5,x_{10},x_{12})$&$(x_6, x_{11}, x_{10})$\\
%    $(x_6,x_6,x_{14})$&$(x_3, x_{11}, x_{14})$\\
%    $(x_7,x_7,x_{12})$&$(x_7, x_{8}, x_{15})$\\
%    $(x_8,x_{13},x_{9})$&$(x_3, x_{14}, x_{8})$\\ 
%    $(x_9,x_{14},x_{15})$&$(x_{7}, x_{14}, x_{12})$\\
%    $(x_{10},x_{13},x_{11})$&$(x_4, x_{12}, x_{15})$\\
%    \hline
%  \end{tabular}
%  \caption{Presentations $T_1$ and $T_{21}$ from \cite{KV2}.}
%  \label{T1T21}
%\end{table}

    Thus these sets of 15 triangles, with angles $\pi/4$, words
    specified in \cite{KV2} written at the boundary and glued together
    respecting orientation, all give a polyhedron that has one vertex
    and the smallest generalised quadrangle as the link.  The
    universal cover of this polyhedron is a hyperbolic triangular
    building \cite{Ballmann-Brin1995}, and the group with 15
    generators $x_1,x_2,\ldots, x_{15}$ and the 15 words from the
    boundaries of the triangles as relations, acts on the building
    cocompactly and simply transitively.

These groups are the first examples of cocompact lattices acting
simply transitively on vertices of hyperbolic triangular Kac-Moody
buildings that are not right-angled. For a general introduction to the
theory of hyperbolic buildings and their lattices, see the survey
\cite{thomas} by A.~Thomas.

Here we study the 23 groups further, motivated by Gromov's famous
surface subgroup question: Does every one-ended hyperbolic group
contain a subgroup which is isomorphic to the fundamental group of a
closed surface of genus at least 2? Recall, that a group $G$ is a {\em
  surface group} if $G=\pi_1(\mathcal{F})$, where $\mathcal{F}$ is a
closed surface. If in addition $\mathcal{F}$ is finite, then
$\pi_1(\mathcal{F})$ has one of the following forms (see
\cite{stillwell})
\begin{enumerate}
\item[(i)] $G=\langle a_1,b_1,\ldots, a_n,b_n\; |\:
  a_1b_1a_1^{-1}b_1^{-1}\cdots a_nb_na_n^{-1}b_n^{-1}\rangle$, when
  $\mathcal{F}$ is orientable and of genus $n$,
\item[(ii)] $G=\langle a_1,\ldots, a_n\; | \: a_1^2a_2^2\cdots
  a_n^2\rangle$, when $\mathcal{F}$ is non-orientable and of genus
  $n$.
\end{enumerate}
 
Gromov's question remains open, but there are many classes of
hyperbolic groups, for which the answer is positive. For example, see
\cite{crisp-sageev-sapir} for surface subgroups of right-angled Artin
groups or \cite{calegari}, where Calegari and Walker show, that a
random group contains many quasiconvex surface subgroups. Existence of
surface subgroups in right-angled hyperbolic buildings was shown in
\cite{futer-thomas}, and in hyperbolic buildings with 4-gonal
apartments in \cite{Vdovina2005}. Existence of surface subgroups in
fundamental groups of higher-dimensional complexes is discussed, for
example, in \cite{reid}.

We are especially interested in periodic apartments, invariant under
an action of a surface group, since such an action implies an
existence of a surface subgroup.  For periodic apartments in Euclidean
buildings using dynamics, see \cite{Ballmann-Brin1995}. In
\cite{Vdovina2005}, periodic apartments were shown to exist in some
hyperbolic buildings.

It is not known, whether the existence of a surface subgroup implies
the existence of a periodic apartment. In this paper we obtain first
candidates for groups not having periodic apartments. We show, that
one cannot find an apartment invariant under a genus two surface group
action in most of the considered 23 groups.

\section{Periodic apartments of genus 2}

\begin{theorem}\label{thm_1}
  There are hyperbolic triangular buildings admitting
  simply-transitive torsion free action and having the smallest
  generalised quadrangle as the link at each vertex that do no have
  any apartments invariant under genus 2 orientable surface group action.
\end{theorem}

\begin{proof}
  Let us assume, that there exists an action of genus 2 surface on an
  apartment of the building. Let's consider the triangulation of the
  surface induced by this action and take the dual graph of this
  triangulation.  It has a vertex for each triangle of which the
  surface is glued together, and edge between two vertices, if the
  corresponding triangles are adjacent. Thus the dual graph is
  3-valent. Since the triangles have angles $\pi/4$, eight of them
  must meet at any vertex of the surface. It means that in the dual
  graph there are cycles of length eight, or, in other words, we can
  think about the surface also as being glued together from octagons
  (see Figure \ref{dualgraph}). Note that the same triangle can appear
  more than once in an octagon, so we should in fact talk about closed
  walks of length 8, but, for simplicity, let us call them
  8-cycles. Since the edges in the dual graph each correspond to an
  edge in the triangulation, the labelling of the edges of the
  triangles with $x_1,\ldots, x_{15}$ corresponds to a colouring of
  the edges in the dual graph.

  \begin{figure}[h!]
    \begin{center}
      \psfrag{xa}{$x_i$} \psfrag{xb}{$x_j$} \psfrag{xc}{$x_k$}
      \includegraphics[width=0.8\textwidth]{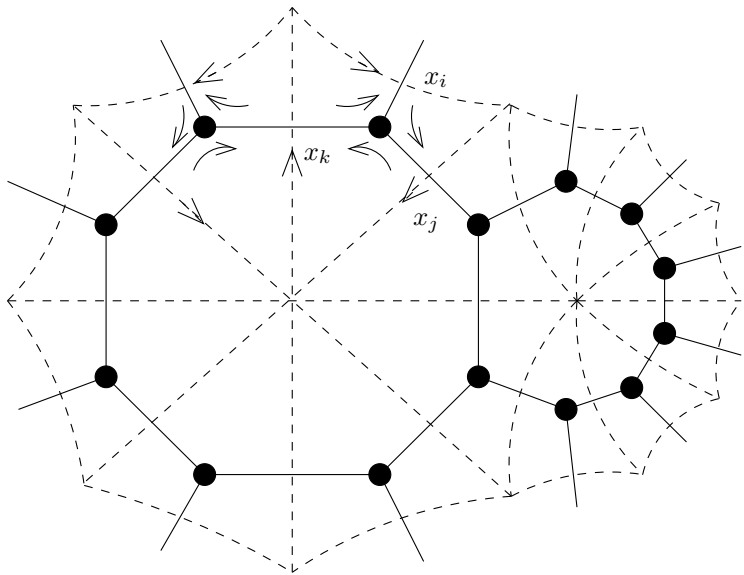}
      \caption{The edges of the dual graph get their labels from the
        sides of the triangles.}
      \label{dualgraph}
    \end{center}
  \end{figure}

  From Euler's formula $V-E+F=2-2g$ we can now deduce the number of
  triangles in a surface of genus 2: The surface is tessellated by
  regular octagons. Three of them meet at each vertex, since the dual
  graph is 3-valent, and each edge is shared by two octagons. Thus if
  we denote the number of octagons by $F$, the number of edges is $4F$
  and the number of vertices is $8F/3$, and we get
  \begin{equation}\label{faces}
    F=6g-6.
  \end{equation}
  So, for a genus 2 surface we need 6 octagons, and thus 16
  triangles. The dual graph therefore has 16 vertices and 24 edges.

  Since the triangles are oriented, they induce an orientation to the
  vertices of the dual graph (Figure \ref{dualgraph}). Two adjacent
  triangles have different orientation, and so also two adjacent
  vertices in the dual graph have different orientation. Thus the dual
  graph must be bipartite. Denote, that the vertices of the 8-cycles
  at the boundaries of the octagons have alternating orientations.

  It is possible that some of the triangles forming the surface are
  glued together from two sides. This means that there is a double
  edge between two vertices in the dual graph. However, let us first
  consider the dual graphs with no double edges. Such graphs have
  girth at least 3.
 
  \subsection{Dual graphs without double edges}
 
  From Gordon Royle's list of cubic graphs \cite{Gordon_web} we see
  that there exist 4060 cubic graphs with 16 vertices, 24 edges and
  girth 3 or more. We generate all of these using \texttt{nauty}
  \cite{nauty}. Only 38 of them are bipartite. By a computer programme
  written in Fortran we check the existence of six 8-cycles in these
  graphs with a depth first algorithm as follows.

  We pick one of the vertices as the starting point for the
  search. There has to be three octagons through this vertex. So we
  can pick any of the adjacent ones to be another vertex in the first
  octagon. Then we proceed along the graph, not visiting the same
  vertex twice, until we arrive to the eighth vertex. If this the one
  we started from, we have an octagon. If not, we go back one step at
  the time, trying all the other possible ways to proceed from the
  previous vertex. We keep track of everything we have tried. When one
  octagon is found, we proceed searching for another 5. Each edge in
  the graph must be used in two different octagons, once to each
  direction.

  Only the graphs that have numbers 3345, 3538, 3621, 4002 and 4060
  when all the 4060 are generated with \texttt{nauty} are bipartite
  and have 6 cycles of length 8 in them, and thus only these five
  graphs fulfill the above conditions for a dual graph of a surface of
  genus 2. Let us call these graphs $G^0_{3345}$, $G^0_{3538}$,
  $G^0_{3621}$, $G^0_{4002}$ and $G^0_{4060}$. Note, that the set of
  six octagons is not unique in any of these five graphs: in the graph
  $G^0_{3345}$ there are 8 ways to pick a set of six octagons with the
  desired properties. In the graph $G^0_{3538}$ there are 2 ways to
  pick the set, in $G^0_{3621}$ 6 ways, in $G^0_{4002}$ 48 ways and in
  the graph $G^0_{4060}$ 18 ways to pick the set of octagons.

  One set of six octagons in the $G^0_{3345}$ (Figure \ref{graph3345})
  is given in Table \ref{octagons1} as a list of vertices.  The cycles
  induce an orientation to the vertices, i.e. following the arrows in
  Figure \ref{graph3345} we obtain the octagons. This orientation has
  to be reversed in every other vertex to obtain the orientation of
  the triangles centered at these vertices of the dual
  graph. Obviously, all the cycles can also be taken to the opposite
  orientation, resulting to the opposite orientation at all vertices.

      \begin{figure}[h!]
        \begin{center}
          \psfrag{1}{$1$} \psfrag{2}{$2$} \psfrag{3}{$3$}
          \psfrag{4}{$4$} \psfrag{5}{$5$} \psfrag{6}{$6$}
          \psfrag{7}{$7$} \psfrag{8}{$8$} \psfrag{9}{$9$}
          \psfrag{10}{$10$} \psfrag{11}{$11$} \psfrag{12}{$12$}
          \psfrag{13}{$13$} \psfrag{14}{$14$} \psfrag{15}{$15$}
          \psfrag{16}{$16$}
          \includegraphics[width=0.80\textwidth]{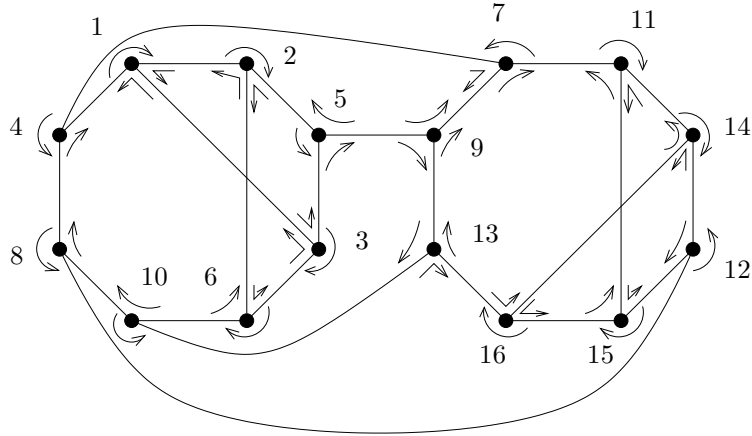}
          \caption{Graph $G^0_{3345}$ with orientations that give six
            8-cycles.}
          \label{graph3345}
        \end{center}
      \end{figure}

      \begin{table}[h!]
        \centering
        \begin{tabular}{| l | l |}
          \hline
          Graph $G^0_{3345}$ & Graph $G^0_{3538}$\\
          \hline
          $(1, 2, 5, 3, 6, 10, 8, 4)$ &$(1, 2, 5, 3, 7, 10, 6, 4)$\\
          $(1, 3, 5, 9, 13, 10,  6, 2)$&$(1, 3, 5, 9, 8, 4, 6, 2)$\\
          $(1, 4, 7, 9, 5, 2, 6, 3)$ &$(1, 4, 8, 12, 15, 11, 7, 3)$\\
          $(4, 8, 12, 14, 16, 15, 11, 7)$ &$(2, 6, 10, 14, 16, 13, 9, 5)$\\
          $(7, 11, 14, 12, 15, 16, 13, 9)$ &$(7, 11, 13, 16, 15, 12, 14, 10)$\\
          $(8, 10, 13, 16, 14, 11, 15, 12)$&$(8, 9, 13, 11, 15, 16, 14, 12)$\\
          \hline
        \end{tabular}
        \caption{Examples of a set of six octagons in the graphs $G^0_{3345}$ and $G^0_{3538}$.}
        \label{octagons1}
      \end{table}
  
      \begin{figure}[h!]
        \begin{center}
          \psfrag{1}{$1$} \psfrag{2}{$2$} \psfrag{3}{$3$}
          \psfrag{4}{$4$} \psfrag{5}{$5$} \psfrag{6}{$6$}
          \psfrag{7}{$7$} \psfrag{8}{$8$} \psfrag{9}{$9$}
          \psfrag{10}{$10$} \psfrag{11}{$11$} \psfrag{12}{$12$}
          \psfrag{13}{$13$} \psfrag{14}{$14$} \psfrag{15}{$15$}
          \psfrag{16}{$16$}
          \includegraphics[width=0.7\textwidth]{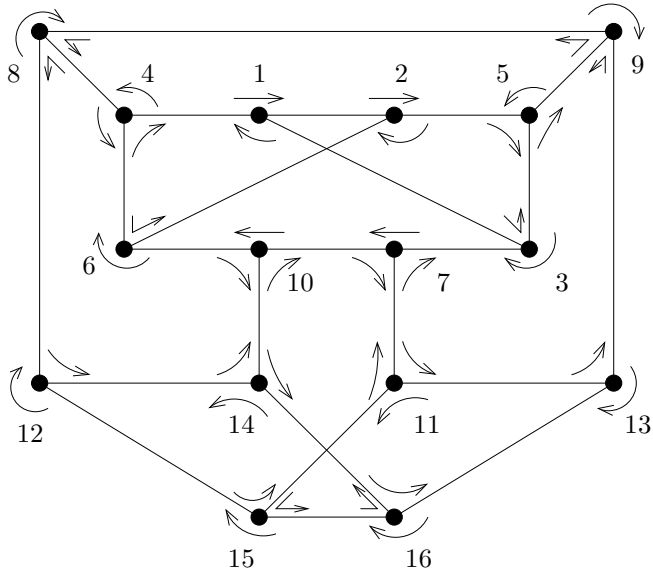}
          \caption{Graph $G^0_{3538}$ with orientations that give six
            8-cycles.}
          \label{graph3538}
        \end{center}
      \end{figure}

      A set of cycles of length 8 in the graph 3538 (Figure
      \ref{graph3538}) are presented in Table \ref{octagons1} and sets
      of cycles in the graphs $G^0_{3621}$ (Figure \ref{graph3621}),
      $G^0_{4002}$ (Figure \ref{graph4002}) and $G^0_{4060}$ (Figure
      \ref{graph4060}) are given in Table \ref{octagons2}. Also in
      these figures is denoted the orientation of the vertices that
      corresponds to the octagons, and the orientations of the
      triangles are obtained reversing the orientation at every other
      vertex.
  
      \begin{table}[h!]
        \centering
        \begin{tabular}{| l | l | l |}
          \hline
          Graph $G^0_{3621}$ & Graph $G^0_{4002}$ & Graph 4060\\
          \hline
          $(1,2,5,9,14,10,6,4)$ &$(1,2,6,11,15,13,7,3)$ &  $(1,2,6,14,9,12,8,3)$\\
          $(1,3,7,11,8,4,6,2)$ &$(1,3,5,10,15,11,8,4)$&$(1,3,7,11,5,12,9,4)$\\
          $(1,4,8,13,15,9,5,3)$ &$(1,4,9,12,16,10,5,2)$&$(1,4,10,15,8,12,5,2)$\\
          $(2,6,10,16,12,7,3,5)$, &$(2,5,3,7,14,16,12,6)$ & $(2,5,11,16,15,10,13,6)$\\
          $(7,12,15,13,16,10,14,11)$ &$(4,8,13,15,10,16,14,9)$&$(3,8,15,16,14,6,13,7)$\\
          $(8,11,14,9,15,12,16,13)$&$(6,12,9,14,7,13,8,11)$&$(4,9,14,16,11,7,13,10)$\\
          \hline
        \end{tabular}
        \caption{Examples of a set of 6 octagons in the graphs $G^0_{3621}$, $G^0_{4002}$ and $G^0_{4060}$.}
        \label{octagons2}
      \end{table}
  
      \begin{figure}[h!]
        \begin{center}
          \psfrag{1}{$1$} \psfrag{2}{$2$} \psfrag{3}{$3$}
          \psfrag{4}{$4$} \psfrag{5}{$5$} \psfrag{6}{$6$}
          \psfrag{7}{$7$} \psfrag{8}{$8$} \psfrag{9}{$9$}
          \psfrag{10}{$10$} \psfrag{11}{$11$} \psfrag{12}{$12$}
          \psfrag{13}{$13$} \psfrag{14}{$14$} \psfrag{15}{$15$}
          \psfrag{16}{$16$}
          \includegraphics[width=0.70\textwidth]{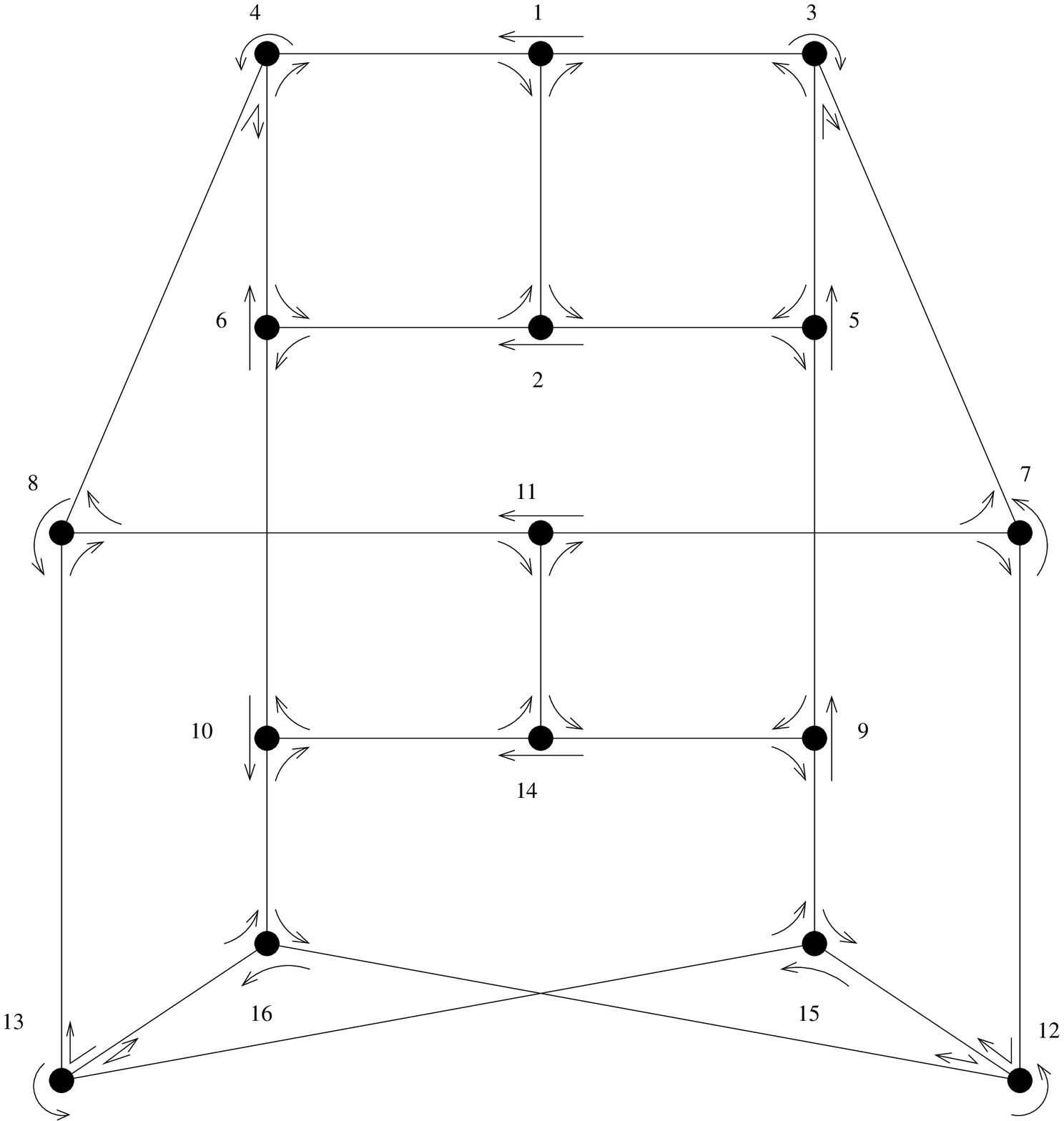}
          \caption{Graph $G^0_{3621}$ with orientations that give six
            8-cycles.}
          \label{graph3621}
        \end{center}
      \end{figure}

      \begin{figure}[h!]
        \begin{center}
          \psfrag{1}{$1$} \psfrag{2}{$2$} \psfrag{3}{$3$}
          \psfrag{4}{$4$} \psfrag{5}{$5$} \psfrag{6}{$6$}
          \psfrag{7}{$7$} \psfrag{8}{$8$} \psfrag{9}{$9$}
          \psfrag{10}{$10$} \psfrag{11}{$11$} \psfrag{12}{$12$}
          \psfrag{13}{$13$} \psfrag{14}{$14$} \psfrag{15}{$15$}
          \psfrag{16}{$16$}
          \includegraphics[width=0.85\textwidth]{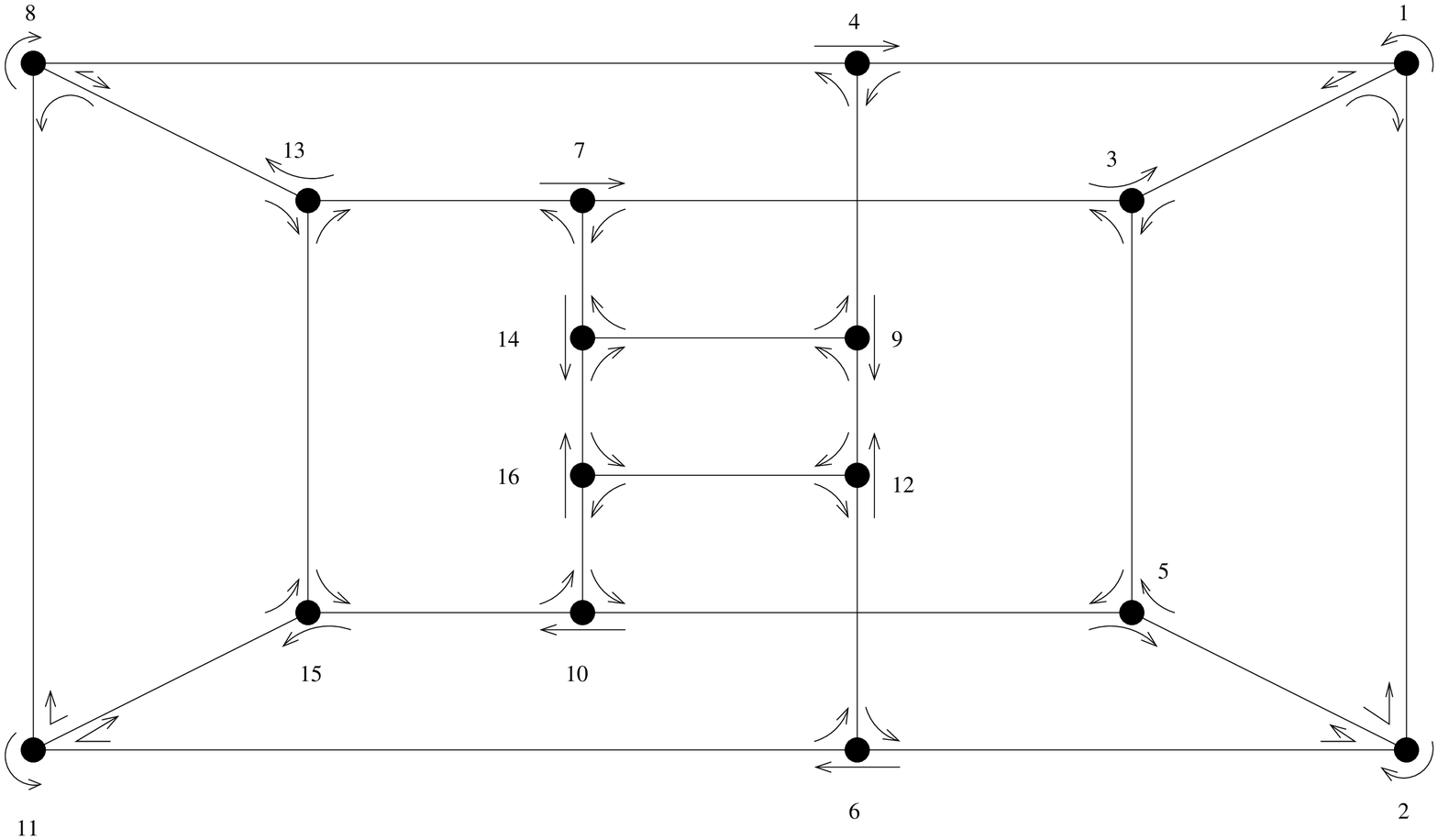}
          \caption{Graph $G^0_{4002}$ with orientations that give six
            8-cycles.}
          \label{graph4002}
        \end{center}
      \end{figure}

      \begin{figure}[h!]
        \begin{center}
          \psfrag{1}{$1$} \psfrag{2}{$2$} \psfrag{3}{$3$}
          \psfrag{4}{$4$} \psfrag{5}{$5$} \psfrag{6}{$6$}
          \psfrag{7}{$7$} \psfrag{8}{$8$} \psfrag{9}{$9$}
          \psfrag{10}{$10$} \psfrag{11}{$11$} \psfrag{12}{$12$}
          \psfrag{13}{$13$} \psfrag{14}{$14$} \psfrag{15}{$15$}
          \psfrag{16}{$16$}
          \includegraphics[width=0.7\textwidth]{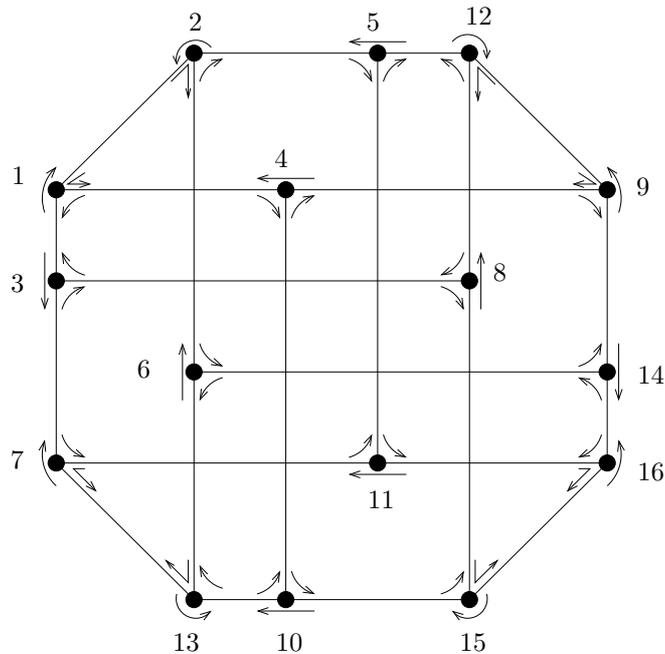}
          \caption{Graph $G^0_{4060}$ with orientations that give six
            8-cycles.}
          \label{graph4060}
        \end{center}
      \end{figure}

      With the help of a another computer programme we then go through
      all the 23 sets of triangles for each of these 5 graphs, and for
      each graph for all the sets of six octagons in it. We look for
      for a labelling of the edges of the dual graphs in such a way
      that around each vertex the labelling of the three edges
      adjacent to it corresponds to a labelling of one of the
      triangles, respecting orientation. Adjacent vertices cannot get
      their labelling from the same triangle, unless the triangle has
      the same label for two edges.

      Let us take for example the graph $G^0_{3621}$ with the 8-cycles
      specified in Table \ref{octagons2} and search for colouring by
      $T_{21}$. There are $15\times 3$ ways to choose the triangle for
      the first vertex and thus the labelling for the three edges
      adjacent to it: 15 triangles to choose from, 3 orientations for
      each one, see Table \ref{presT} for the triangles. Let us pick
      $(x_1,x_5,x_2)$ as in Figure \ref{colour_no}. Now, for vertex 2,
      we have two choices: $(x_1,x_6,x_4)$ or $(x_1,x_{13},x_3)$. Let
      us take the first one. For vertex 3 can use either
      $(x_5,x_9,x_{10})$ or $(x_5,x_{13},x_9)$. If we chose the first
      one, for vertex 5 we would need a triplet with $x_{9}$ followed
      by $x_4$.  If we chose the latter one, we would need a triplet
      with $x_{13}$ followed by $x_6$. However, there are no such
      triangles in $T_{21}$, see Table \ref{presT}. So, our choice for
      the vertex 2 leads us nowhere. We had also another choice for
      vertex 2, namely $(x_1,x_{13},x_3)$. But, again, choosing either
      of our possibilities $(x_5,x_9,x_{10})$ or $(x_5,x_{13},x_9)$
      for vertex 3 makes it impossible to get a triangle for vertex
      5. Thus, the initial choice for the first vertex does not lead
      to any colouring of the graph.

      \begin{figure}[h!]
        \begin{center}
          \psfrag{1}{$1$} \psfrag{2}{$2$} \psfrag{3}{$3$}
          \psfrag{4}{$4$} \psfrag{5}{$5$} \psfrag{6}{$6$}
          \psfrag{?}{?}  \psfrag{x1}{$x_1$} \psfrag{x2}{$x_2$}
          \psfrag{x4}{$x_4$} \psfrag{x5}{$x_5$} \psfrag{x6}{$x_6$}
          \psfrag{x10}{$x_{10}$} \psfrag{x9}{$x_9$} \psfrag{orx9}{or
            $x_{9}$} \psfrag{orx13}{or $x_{13}$}
          \includegraphics[width=0.70\textwidth]{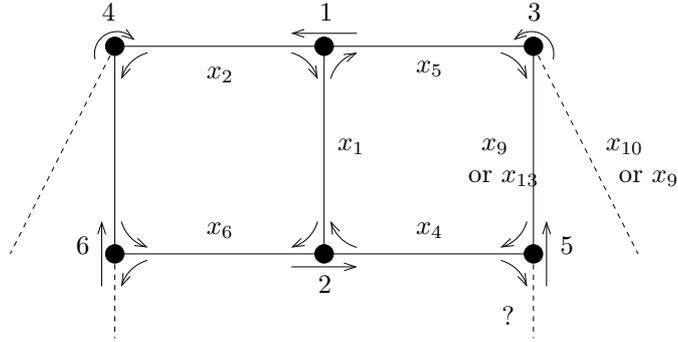}
          \caption{Colouring vertex 1 in the graph $G^0_{3621}$ with
            $(x_1,x_5,x_2)$, 2 with $(x_1,x_6,x_4)$ and 3 with either
            $(x_5,x_9,x_{10})$ or $(x_5,x_{13},x_9)$ leaves no
            possibility to colour vertex 5.}
          \label{colour_no}
        \end{center}
      \end{figure}

      With the computer program we try out all possible choices: we
      try out all 23 sets of triangles for all five candidates for
      dual graphs, for each of them all ways of picking the six
      octagons. In each case we try out all 45 choices for the labels
      around the first vertex. As a result, we find colourings for the
      graph $G^0_{3345}$ by the triangles in the presentations $T_1$
      and $T_2$ listed in \cite{KV2}, but no colouring with any subset
      of the triangles in any of the other presentations $T_3$ --
      $T_{23}$.

      \begin{figure}[ht!]
        \begin{center}
          \psfrag{x1}{$x_1$} \psfrag{x2}{$x_2$} \psfrag{x3}{$x_3$}
          \psfrag{x4}{$x_4$} \psfrag{x9}{$x_9$} \psfrag{x10}{$x_{10}$}
          \psfrag{x11}{$x_{11}$} \psfrag{x13}{$x_{13}$}
          \psfrag{x14}{$x_{14}$} \psfrag{x15}{$x_{15}$}
          \includegraphics[width=0.9\textwidth]{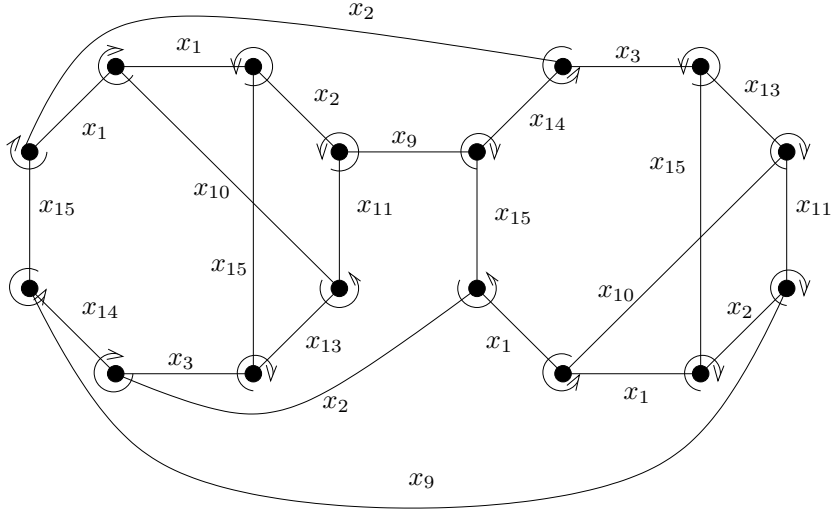}
          \caption{Colouring of the graph $G^0_{3345}$ with triangles from
            $T_1$.}
          \label{colour_yes}
        \end{center}
      \end{figure}

      A colouring of the graph $G^0_{3345}$ with triangles from $T_1$ is
      presented in Figure \ref{colour_yes}. Thus these 16 triangles,
      namely $(x_1,x_1,x_{10})$ at vertices $1$ and $16$ of the graph
      $G^0_{3345}$ (see Table \ref{presT} for the triangles and Figure
      \ref{graph3345} for the labelling of the vertices),
      $(x_1,x_{15},x_2)$ at vertices 2, 4, 13 and 15,
      $(x_2,x_{11},x_9)$ at vertices 5 and 12, $(x_2,x_{14},x_3)$ at
      vertices 7 and 10, $(x_3,x_{15},x_{13})$ at vertices 6 and 11,
      $(x_9,x_{14},x_{15})$ at vertices 8 and 9 and
      $(x_{10},x_{13},x_{11})$ at vertices 3 and 14, give a surface of
      genus 2. Since these 7 triangles used in the colouring are among
      the 15 triangles of the presentation $T_2$ as well, see
      \cite{KV2}, the same periodic plane exists in $T_2$, too. For
      the other four candidates for dual graphs no colourings can be
      found.

      Thus in the buildings defined by the presentations $T_1$ and
      $T_2$ there are periodic planes of genus 2 and therefore surface
      subgroups of the same genus. However, when considering the
      possible dual graphs that do not have multiple edges, there are
      no apartments invariant by genus 2 surface group actions
      associated to the other 21 triangular presentations.

      \subsection{Dual graphs with double edges}

      Let us then consider the surfaces where some of the triangles
      are glued together by two sides. This means that in the dual
      graph there are two edges between some two vertices. There cannot however be adjacent double edges, that would give degree 4
      to the vertex between them. Also seven double edges altogether
      would make the existence of cycles of length 8 impossible. Thus, we will generate and check all possible dual graphs with six or less double edges.
      
      We generate the possible dual graphs with double edges as
      follows: For the dual graphs with $n$ double edges we first
      generate with \texttt{nauty} \cite{nauty} the connected, bipartite graphs
      with 16 vertices and $24-n$ edges, with vertices of degree two and
      three. Then we check whether the vertices of degree two in a
      graph are pairwise adjacent. If they are, we double
      these edges to obtain a 3-valent graph with $n$ double
      edges. See Table \ref{tab:graphnbs} for the numbers of graphs obtained.
      
       \begin{table}[h!]
        \begin{center}
          \begin{tabular}{ | c | c | c | c | p{5cm}|}
            \hline
            Double edges & Graphs & Possible dual graphs \\ \hline
            0 &   38 &  $G^0_{3345}$, $G^0_{3538}$, $G^0_{3621}$, $G^0_{4002}$, $G^0_{4060}$\\
            1 & 86 &  $G^1_{61}$, $G^1_{84}$\\
            2 & 145 & $G^2_{20}$, $G^2_{25}$, $G^2_{78}$, $G^2_{84}$\\
            3 & 132 &  $G^3_{112}$\\
            4 & 75 &  -\\
            5 & 21 &  -\\
            6 & 1 & -\\
            \hline
          \end{tabular}
        \end{center}
        \caption{The amounts of graphs with different numbers of double edges}
        \label{tab:graphnbs}
      \end{table}

     After generating the graphs we run the same depth first
      search as earlier to see, whether the graphs consist of six cycles
      of length eight. We obtain two possible dual graphs with
      one double edge, nine with two double edges, four with three and
      four with four double edges. With more double edges suitable graphs do
      not exist.

      When studying these graphs further, we see that in fact not even
      all of these graphs are possible dual graphs. Namely, the
      8-cycles in the dual graphs arise from orientations given to the
      vertices. Thus, if we have a double edge between two vertices,
      there is only two possibilities how a cycle can go through the
      vertices. When the graph is drawn on a plane, either both the vertices are
      oriented to the same direction as in Figure \ref{double_edge1},
      or, they have opposite orientations, as in Figure
      \ref{double_edge2}. In the first case a cycle coming in from
      vertex $C$ towards vertex $A$ would continue to vertex $B$ and
      $D$ and then further. Similarly for a path coming from $D$
      towards $B$. With these two paths the edges $AC$ and $BD$ are
      travelled twice, but the edges between $A$ and $B$ are
      not. Thus, the vertices joined by a double edge must have
      opposite orientations to allow all edges to be travelled twice. 

      \begin{figure}[ht!]
        \psfrag{A}{$A$} \psfrag{B}{$B$} \psfrag{C}{$C$}
        \psfrag{D}{$D$} \centering
        \begin{subfigure}[b]{0.45\textwidth}
          \includegraphics[width=\textwidth]{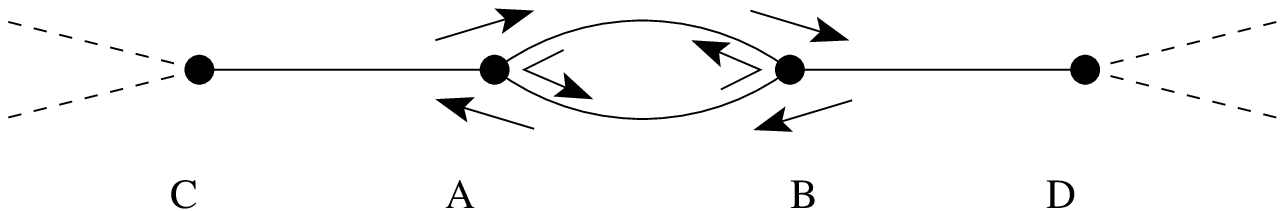}
          \caption{Vertices with same orientation}
          \label{double_edge1}
        \end{subfigure}
        ~ 
        \begin{subfigure}[b]{0.45\textwidth}
          \includegraphics[width=\textwidth]{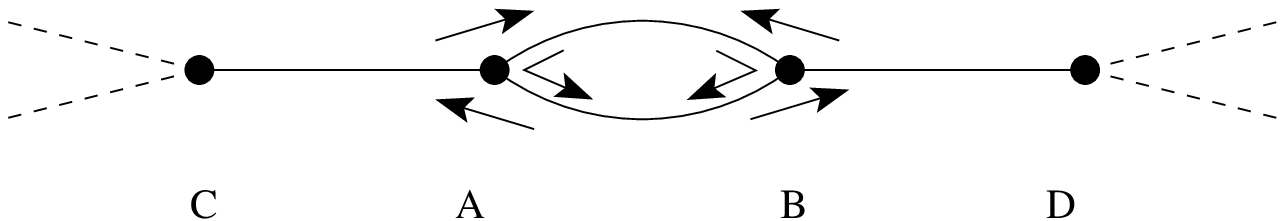}
          \caption{Vertices with opposite orientations}
          \label{double_edge2}
        \end{subfigure}
        \caption{Possible orientations of vertices joined by a double
          edge}\label{double_edge}
      \end{figure}

      Now, when the vertices joined by a double edge have opposite
      orientations, as in Figure \ref{double_edge2}, we immediately
      have some additional information about the 8-cycles. Namely, a
      path from $C$ towards $A$ continues to $B$ and then back to
      $A$. Thus, to get a cycle, this paths needs to continue with a
      cycle of four edges from $C$. Similarly for a path from $D$
      towards $B$.  In half of the graphs that are found to have six
      8-cycles with the depth first search there are no such cycles of
      length four available, and thus they are not suitable for dual
      graphs. Let us call the graphs we have left by $G^d_n$, where
      $d$ is the number of double edges, and $n$ is the number of the
      graph in the list of graphs generated by \texttt{nauty}. Thus we have the
      two graphs with one double edge, $G^1_{61}$ and $G^1_{84}$, four
      graphs with two double edges, $G^2_{20}$, $G^2_{25}$, $G^2_{78}$
      and $G^2_{84}$, and one graph $G^3_{112}$ with three double
      edges, see Table \ref{tab:graphnbs}. The graphs are presented in Figures \ref{Gd1}, \ref{Gd2} and
      \ref{Gd3_112}.

\begin{figure}[ht!]
  \centering \psfrag{1}{$1$} \psfrag{2}{$2$} \psfrag{3}{$3$}
  \psfrag{4}{$4$} \psfrag{5}{$5$} \psfrag{6}{$6$} \psfrag{7}{$7$}
  \psfrag{8}{$8$} \psfrag{9}{$9$} \psfrag{10}{$10$} \psfrag{11}{$11$}
  \psfrag{12}{$12$} \psfrag{13}{$13$} \psfrag{14}{$14$}
  \psfrag{15}{$15$} \psfrag{16}{$16$}
   \begin{subfigure}[b]{0.43\textwidth}
       \includegraphics[width=\textwidth]{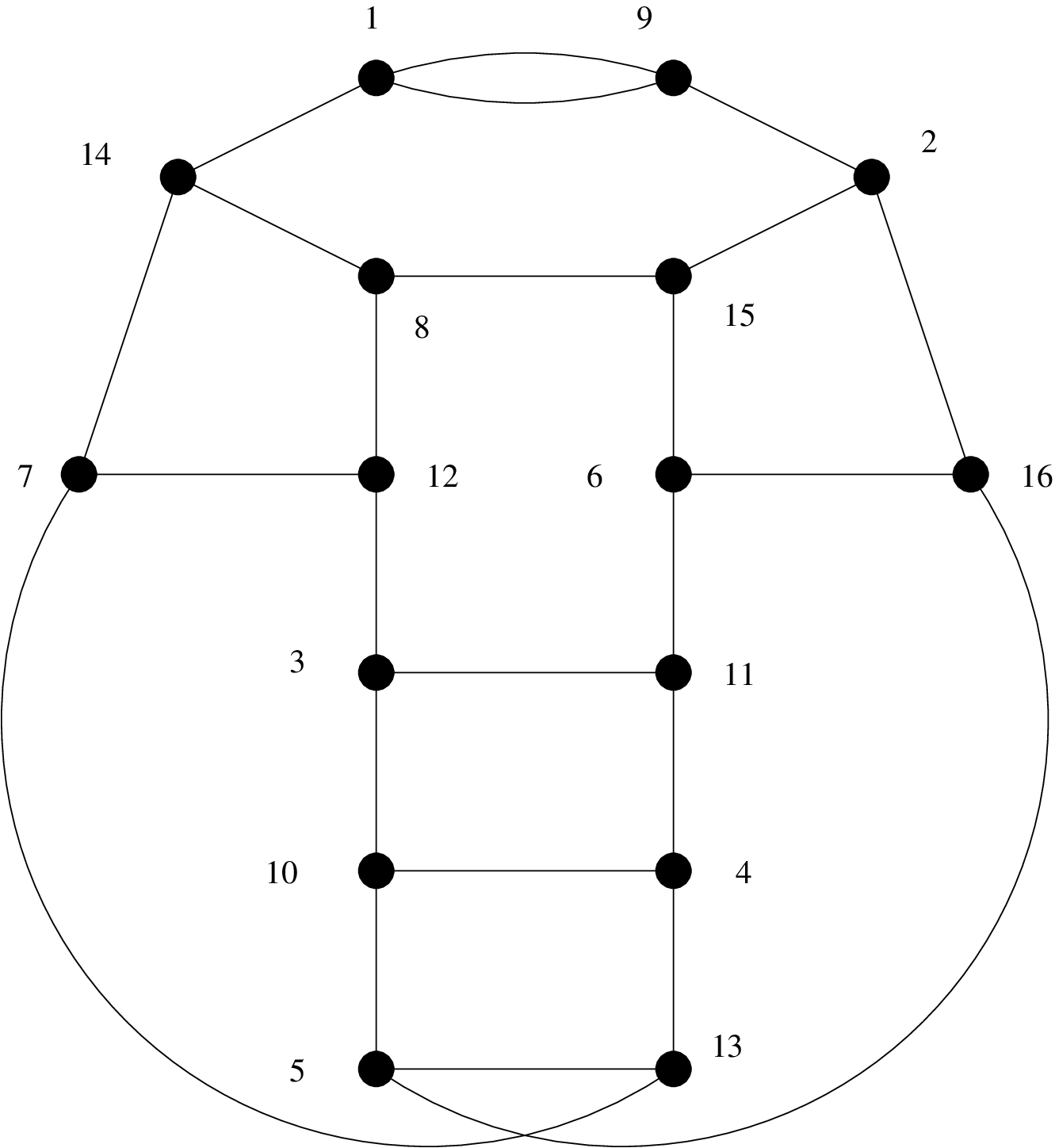}
       \caption{Graph $G^1_{61}$}
   \end{subfigure}
   ~ 
   \begin{subfigure}[b]{0.45\textwidth}
       \includegraphics[width=\textwidth]{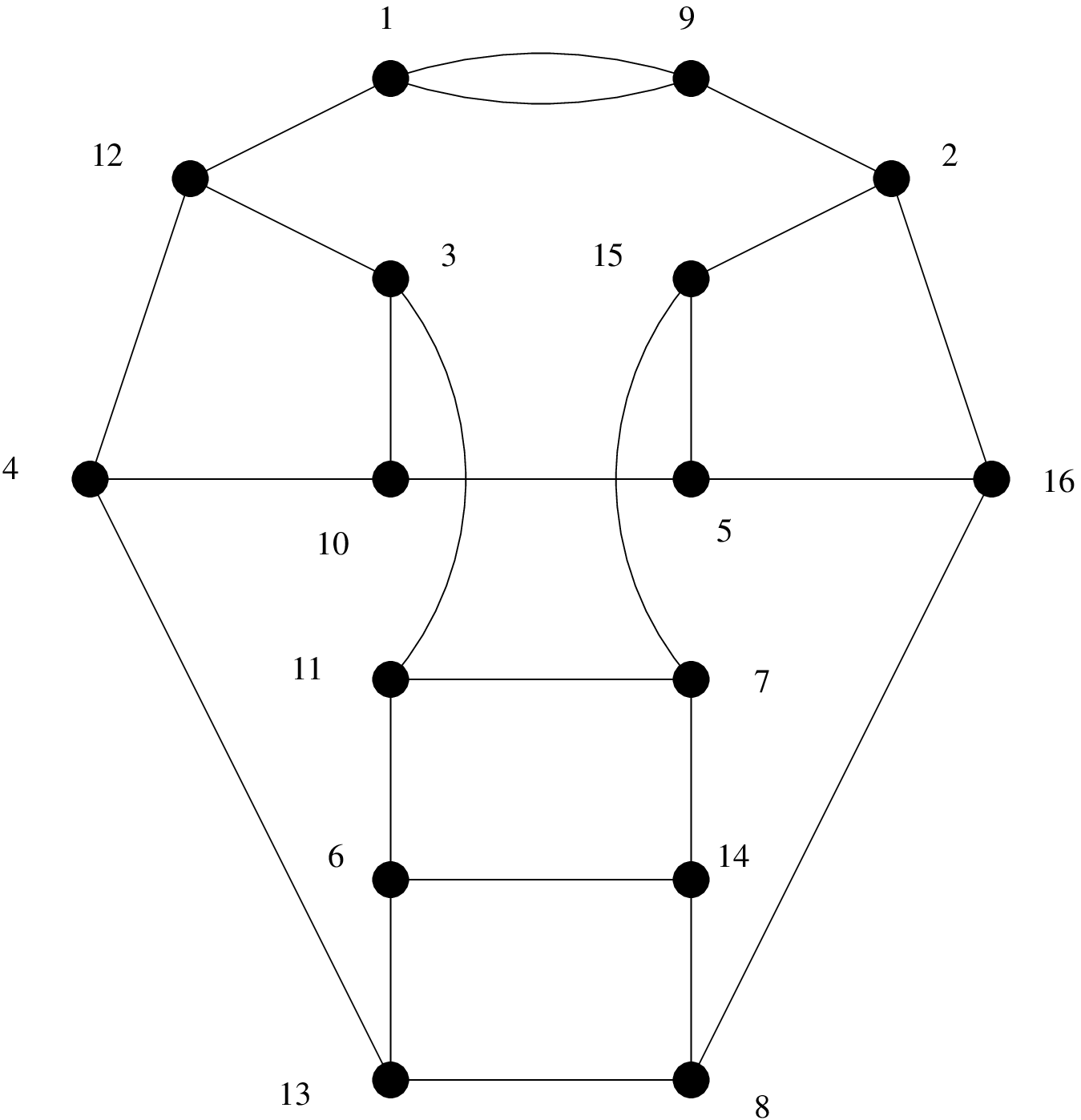}
       \caption{Graph $G^1_{84}$ }
   \end{subfigure}
   \caption{Graphs with one double edge}\label{Gd1}
\end{figure}

\begin{figure}[ht!]
  \centering \psfrag{1}{$1$} \psfrag{2}{$2$} \psfrag{3}{$3$}
  \psfrag{4}{$4$} \psfrag{5}{$5$} \psfrag{6}{$6$} \psfrag{7}{$7$}
  \psfrag{8}{$8$} \psfrag{9}{$9$} \psfrag{10}{$10$} \psfrag{11}{$11$}
  \psfrag{12}{$12$} \psfrag{13}{$13$} \psfrag{14}{$14$}
  \psfrag{15}{$15$} \psfrag{16}{$16$}
   \begin{subfigure}[b]{0.43\textwidth}
       \includegraphics[width=\textwidth]{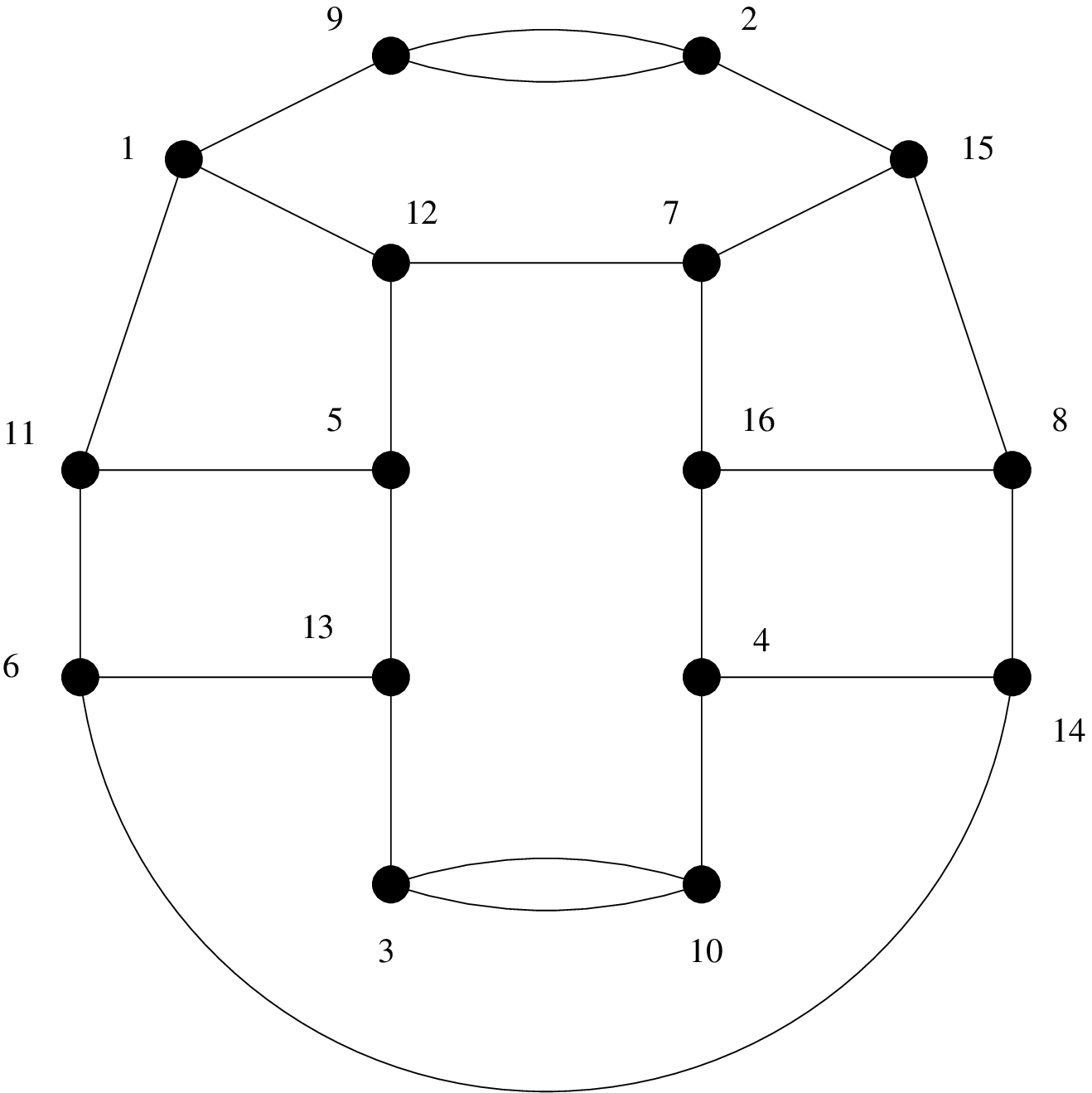}
       \caption{Graph $G^2_{20}$}
   \end{subfigure}
   \quad 
   \begin{subfigure}[b]{0.45\textwidth}
       \includegraphics[width=\textwidth]{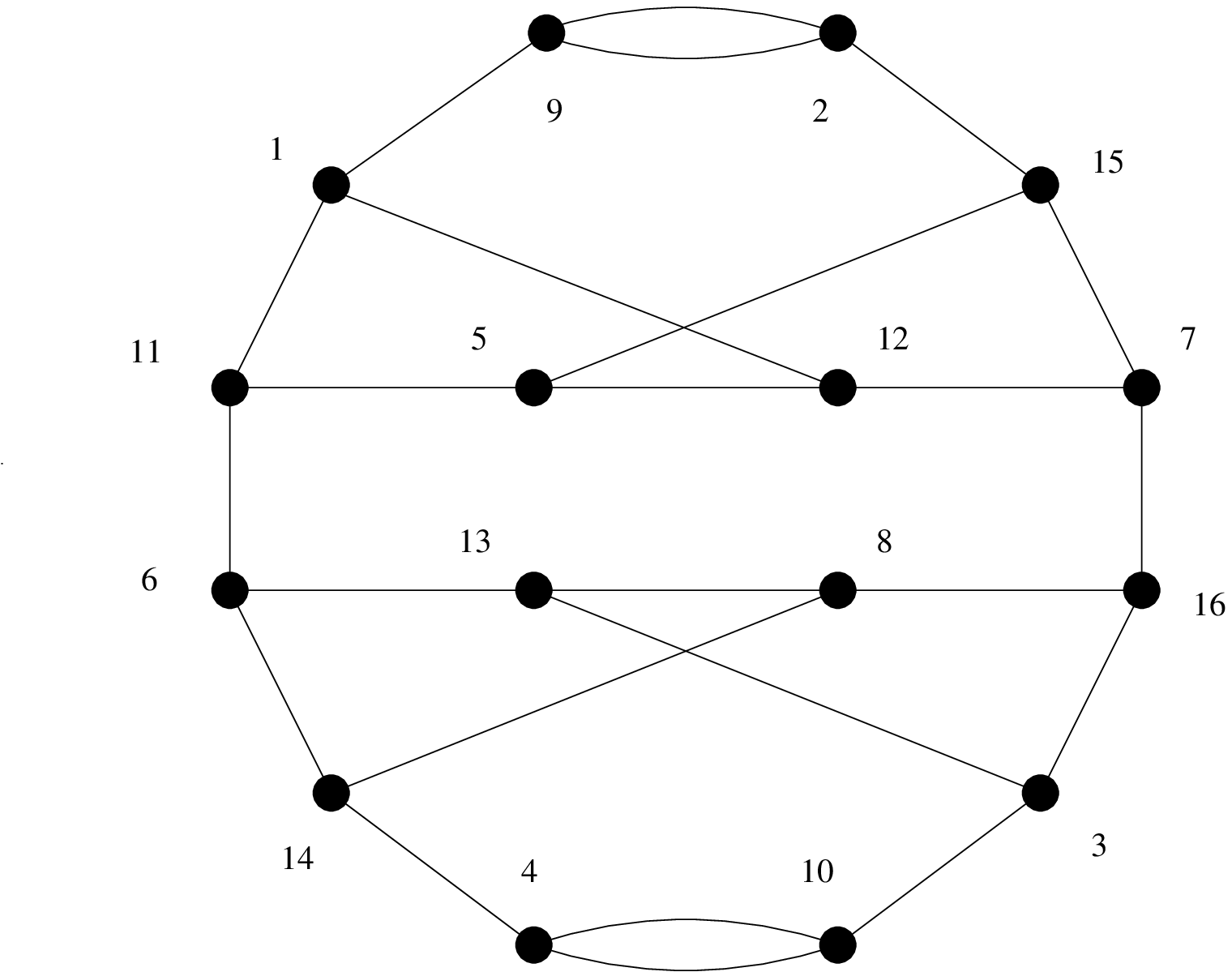}
       \caption{Graph $G^2_{25}$ }
   \end{subfigure}\\
   \bigskip
      \begin{subfigure}[b]{0.45\textwidth}
       \includegraphics[width=\textwidth]{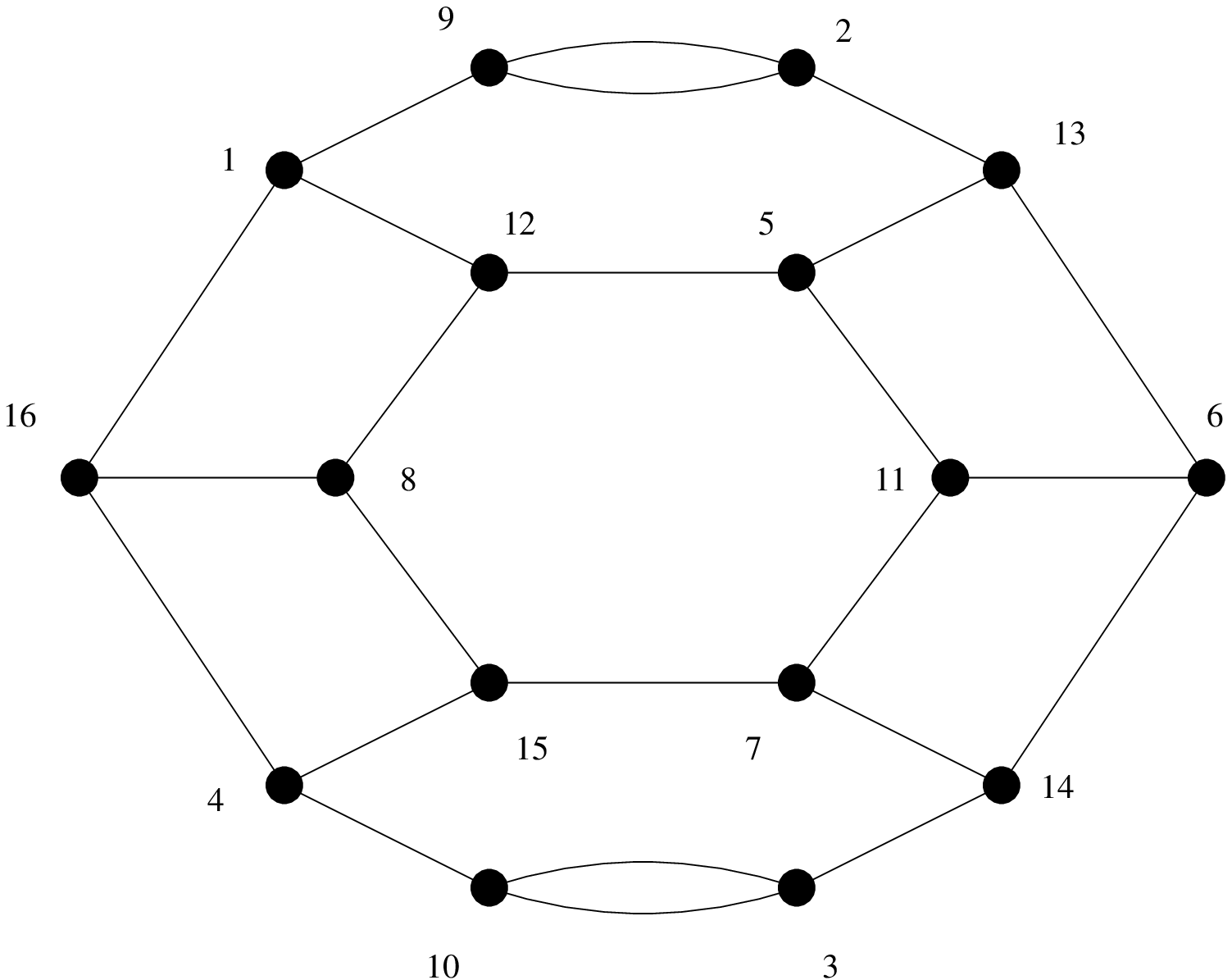}
       \caption{Graph $G^2_{78}$}
   \end{subfigure}
   \quad
   \begin{subfigure}[b]{0.42\textwidth}
       \includegraphics[width=\textwidth]{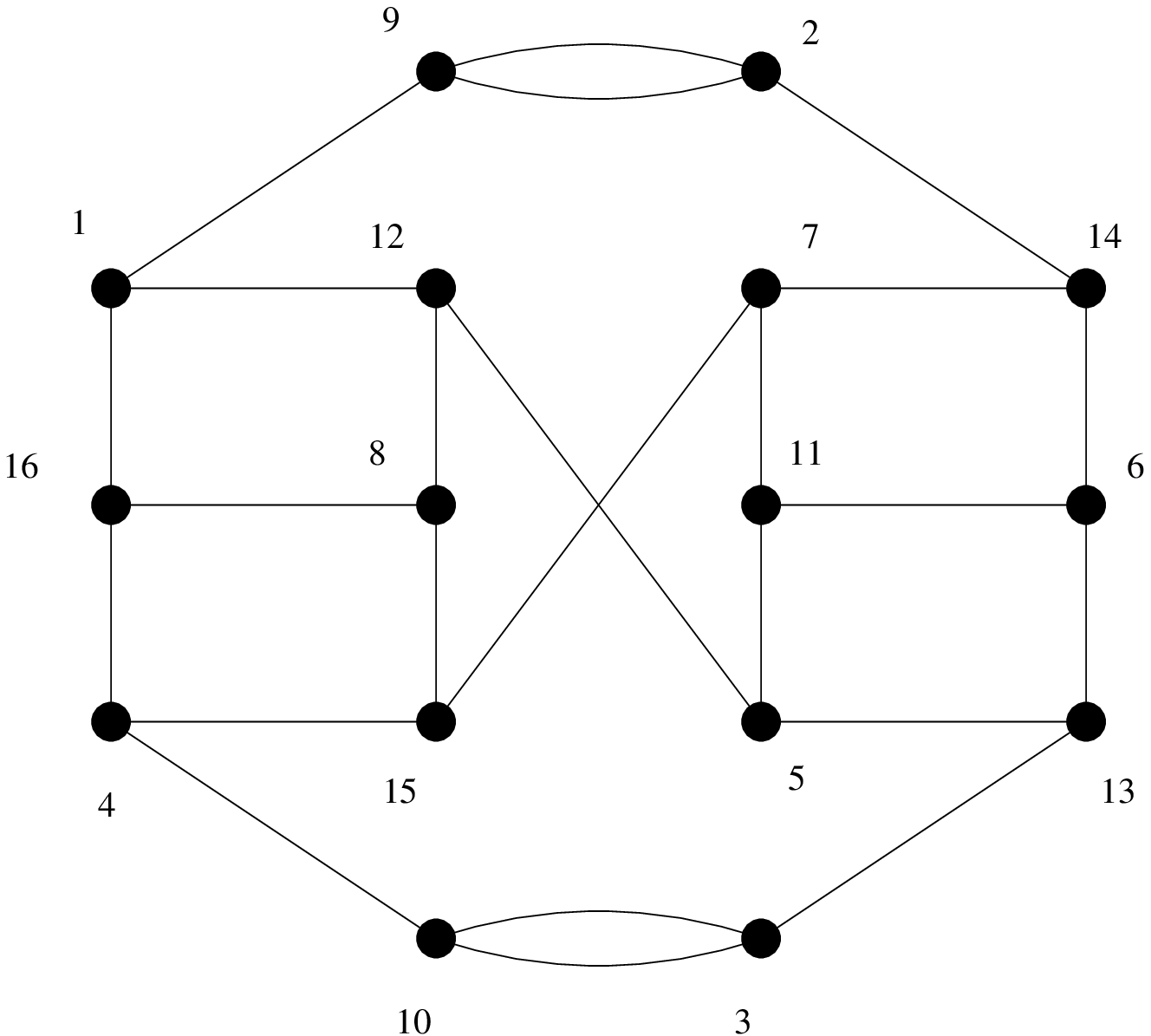}
       \caption{Graph $G^2_{84}$ }
   \end{subfigure}
   \caption{Graphs with two double edges}\label{Gd2}
\end{figure}

 \begin{figure}[ht!]
    \begin{center}
      \psfrag{1}{$1$} \psfrag{2}{$2$} \psfrag{3}{$3$} \psfrag{4}{$4$}
      \psfrag{5}{$5$} \psfrag{6}{$6$} \psfrag{7}{$7$} \psfrag{8}{$8$}
      \psfrag{9}{$9$} \psfrag{10}{$10$} \psfrag{11}{$11$}
      \psfrag{12}{$12$} \psfrag{13}{$13$} \psfrag{14}{$14$}
      \psfrag{15}{$15$} \psfrag{16}{$16$}
  \includegraphics[width=0.65\textwidth]{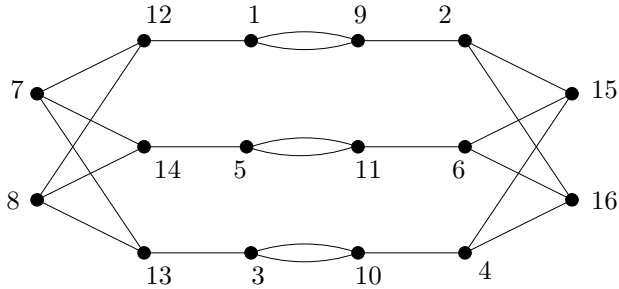}
      \caption{Graph $G3_{112}$ with three double edges.}
      \label{Gd3_112}
    \end{center}
  \end{figure}

  All of these seven graphs have several possible ways to pick the set of
  six 8-cycles. We check by computer for each graph all $2^{16}$
  possibilities to orient the vertices in order to see which
  orientations produce a set of six 8-cycles. As a result we have 8
  different sets of six 8-cycles for $G^1_{61}$, 24 for $G^1_{84}$, 64
  for $G^2_{20}$ and $G^2_{25}$, 32 for $G^2_{78}$ and $G^2_{84}$, and
  1024 for the graph $G^3_{112}$.

  Then we search for colourings of the vertices of these seven
  multigraphs with triangles from the 23 groups in \cite{KV2}, giving to the three
  edges meeting at a vertex labels from the three sides of one of the
  triangles. The search is otherwise similar to earlier, but this time
  we first color the graphs without paying attention to orientations
  of the vertices and triangles. Only after receiving a candidate for
  possible coloring we then check, whether the coloring corresponds to
  one of the orientations at the vertices that gives six 8-cycles in
  the graph.

  We obtain colourings for the graphs $G^2_{78}$, $G^2_{84}$ and
  $G^3_{112}$, the other four multigraphs do no admit any
  colourings. The graphs $G^2_{78}$ and $G^2_{84}$ can be coloured
  with the triangles from the group $T_{18}$ in \cite{KV2}. The graph
  $G^3_{112}$ can be coloured with triangles from the groups $T_1$,
  $T_7$ or $T_9$. We already knew that for $T_1$ a periodic
  apartment exists in the corresponding building. However, this
  colouring of $G^3_{112}$ by $T_1$ uses different triangles than the colouring obtained for $G^0_{3345}$, and 
  this periodic apartment is not found in the building corresponding
  to $T_2$. The new cases of periodic apartments with dual
  graph $G^3_{112}$ are in the buildings given by $T_7$ and
  $T_9$. The obtained colourings of $G^2_{84}$ by $T_{18}$ and of
  $G^3_{112}$ by $T_7$ and $T_9$ are presented in the Table
  \ref{tableT7T9}: on line $i$ is the triangle that colours vertex $i$
  of the corresponding graph, when the vertices are numbered as in the Figures \ref{Gd2} and
      \ref{Gd3_112}. The labels on the sides of the cyclic
  triangle are given in the order matching to the ascending order of
  the labels of the vertices adjacent to $i$.

      \begin{table}[h!]
        \centering
        \begin{tabular}{| l | l | l |}
          \hline
          Colouring with $T_7$ & Colouring with $T_9$ & Colouring with $T_{18}$\\
          \hline
          $(x_{7},x_{3},x_{4})$&$(x_{14}, x_{14}, x_{4})$ &$(x_{15}, x_{13}, x_{5})$\\
          $(x_{6},x_{9},x_{11})$&$(x_{4}, x_{8}, x_{5})$ &$(x_{1}, x_{1}, x_{15})$\\
          $(x_{7},x_{3},x_{6})$&$(x_{15}, x_{9}, x_{6})$ &$(x_{1}, x_{1}, x_{15})$\\
          $(x_{4},x_{15},x_{13})$&$(x_{13}, x_{3}, x_{10})$ &$(x_{15}, x_{10}, x_{12})$\\
          $(x_{12},x_{12},x_{3})$&$(x_{15}, x_{9}, x_{13})$ &$(x_{7}, x_{11}, x_{10})$\\
          $(x_{3},x_{7},x_{6})$&$(x_{6}, x_{7}, x_{12})$ &$(x_{6}, x_{12}, x_{5})$\\
          $(x_{13},x_{11},x_{6})$&$(x_{8}, x_{7}, x_{3})$ &$(x_{9}, x_{13}, x_{11})$\\
          $(x_{15},x_{9},x_{7})$&$(x_{5}, x_{12}, x_{10})$ &$(x_{9}, x_{7}, x_{6})$\\
          $(x_{7},x_{3},x_{6})$&$(x_{14}, x_{14}, x_{4})$ &$(x_{15}, x_{1}, x_{1})$\\
          $(x_{7},x_{3},x_{4})$&$(x_{15}, x_{9}, x_{13})$ &$(x_{1}, x_{1}, x_{15})$\\
          $(x_{12},x_{12},x_{3})$&$(x_{15}, x_{9}, x_{6})$ &$(x_{7}, x_{6}, x_{9})$\\
          $(x_{4},x_{13},x_{15})$&$(x_{4}, x_{8}, x_{5})$ &$(x_{13}, x_{11}, x_{9})$\\
          $(x_{6},x_{11},x_{9})$&$(x_{6}, x_{7}, x_{12})$ &$(x_{15}, x_{10}, x_{12})$\\
          $(x_{3},x_{6},x_{7})$&$(x_{13}, x_{3}, x_{10})$ &$(x_{15}, x_{5}, x_{13})$\\
          $(x_{9},x_{15},x_{7})$&$(x_{8}, x_{3}, x_{7})$ &$(x_{10}, x_{11}, x_{7})$\\
          $(x_{11},x_{13},x_{6})$&$(x_{5}, x_{10}, x_{12})$ &$(x_{5}, x_{12}, x_{6})$\\
          \hline
        \end{tabular}
        \caption{Colours for the vertices of $G^3_{112}$ from presentations $T_7$ and $T_9$  and for the vertices of $G^2_{84}$ from $T_{18}$ in \cite{KV2}.}
        \label{tableT7T9}
      \end{table}

      We have now went through all possible colourings for all possible
      dual graphs both with and without multiple edges. As a result we
      found periodic apartments of genus 2 and therefore surface
      subgroups in the buildings given by $T_1$, $T_2$, $T_7$, $T_9$ and
      $T_{18}$ in \cite{KV2}. For the other 18 cases presented in
      \cite{KV2} no periodic apartment invariant under a genus 2
      action exists. This ends the proof of Theorem \ref{thm_1}.
    \end{proof}

\begin{remark}
  From \eqref{faces} we note immediately that surfaces of genus 0 or 1 are
  impossible.
\end{remark}

\begin{remark}
  The existence of periodic apartments of genus 3 could in theory be
  checked the same way. For them the possible dual graph would have to
  be bipartite trivalent graphs with 32 vertices, 48 edges and 12
  cycles of length 8. However, since there are already 18941522184590
  trivalent graphs with 32 vertices without double edges
  \cite{Gordon_web}, the calculation time to find the possible graphs
  and to search colourings with the current algorithms would be too long.
\end{remark}

\section*{Acknowledgements}
The authors would like to thank Tim Steger for useful
discussions. This article was finalized in July 2015 when the authors
were invited to work at Max Planck Institute for Mathematics in
Bonn. We wish to thank MPIM for their hospitality. We also acknowledge the support of the 
EPSRC grant EP/K016687/1.

\end{document}